\documentclass[11pt,twoside]{amsart}
\usepackage{amsmath, amsthm, amscd, amsfonts, amssymb, graphicx, color}
\usepackage[bookmarksnumbered, plainpages]{hyperref}
\usepackage{longtable}
\usepackage{hyperref} 
\usepackage{cleveref} 
\newtheorem{theorem}{Theorem}[section]
\newtheorem{thm}[theorem]{Theorem}
\newtheorem{lem}[theorem]{Lemma}

\theoremstyle{definition}

\newtheorem{example}[theorem]{Example}

\newtheorem{conj}[theorem]{Conjecture}
\theoremstyle{remark}

\numberwithin{equation}{section}

\begin{document}
\title[Supercharacter Theories of a  Finite Group]{An Algorithm for Constructing All Supercharacter Theories of a  Finite Group}

\author[A. R. Ashrafi]{A. R. Ashrafi}

\address{Department of Pure Mathematics, Faculty of Mathematical Sciences, University of Kashan, Kashan 87317$-$53153, I. R. Iran}
\curraddr{}
\email{ashrafi@kashanu.ac.ir}

\thanks{The first author is supported by the University of Kashan under grant number 785149/7.}

\author[L. Ghanbari-Maman]{L. Ghanbari-Maman}
\address{Department of Computer Science, Faculty of Mathematical Sciences,
University of Kashan, Kashan 87317$-$53153, I. R. Iran}
\curraddr{Department of Bioinformatics, Institute of Biochemistry and Biophysics, University of Tehran, Tehran, Iran}
\email{leila.ghanbari@ut.ac.ir}
\thanks{}
\author[K. Kavousi]{K. Kavousi}
\address{Department of Bioinformatics, Institute of Biochemistry and Biophysics, University of Tehran, Tehran, Iran}
\curraddr{}
\email{kkavousi@ut.ac.ir}
\thanks{}

\author[F. Koorepazan-Moftakhar]{F. Koorepazan-Moftakhar}
\address{Department of Pure Mathematics, Faculty of Mathematical Sciences, University of Kashan, Kashan 87317$-$53153, I. R. Iran}
\curraddr{Department of Mathematical Sciences, Sharif University of Technology, Azadi Street, P.O. Box 11155-9415, Tehran, Iran}
\email{f.moftakhar@sharif.edu}
\thanks{}
\date{}

\dedicatory{}

\begin{abstract}
In 2008, Diaconis and Isaacs introduced the notion of a supercharacter theory of a finite group in which supercharacters replace with irreducible characters and superclasses by conjugacy classes. In this paper, we introduce an algorithm for constructing supercharacter theories of a finite group by which all supercharacter theories of groups containing up to $14$ conjugacy classes are calculated. 
\end{abstract}

\maketitle

\section{Introduction}
Suppose $UT_{n}(q)$ denotes the set of all $n \times n$ unipotent upper-triangular matrices over the finite field $GF(q)$. While working on the complex characters of this group, Andr\'{e} constructed something nowadays called a supercharacter theory\cite{2,3,4}. Diaconis and Isaacs in their seminal paper\cite{6},  axiomatized the notion of supercharacter theories of finite groups. To define, we assume that $G$ is a finite group, $Irr(G)$ denotes the set of all ordinary irreducible characters of $G$ and  $Con(G)$
is the set of all conjugacy classes of $G$. A pair $(\mathcal{X}, \mathcal{K})$ is a supercharacter theory of $G$ if the following conditions hold:
\begin{enumerate}
\item  $ \mathcal{X}$ and $\mathcal{K}$ are set partitions of $Irr(G)$ and $Con(G)$, respectively;
\item  $ \mathcal{K} $ contains $\{e\}$, where $e$ denotes the identity element of $G$;
\item  $|\mathcal{X}| = |\mathcal{K}|$;
\item  For every $ X \in \mathcal{X}$, the characters $\sigma_{X} = \sum_{\chi \in X}\chi(e)\chi$ are constant on each $K \in \mathcal{K}$.
\end{enumerate}
The characters $\sigma_{X}$ are called \textit{supercharacters}, and the members of $\mathcal{K}$ are called \textit{superclasses} of $G$\cite{6}. Throughout this paper, $Sup(G)$ denotes the set of all supercharacter theories of $G$. Now, let $\mathcal{X}$ = $\{ \{1_G\},  Irr(G) \setminus \{ 1_G\}\}$ and $\mathcal{K}$ = $\{\{e\},Con(G) \setminus\{ e\}\}$, then $m(G) = (Irr(G),Con(G))$ and
$M(G) = (\mathcal{X}, \mathcal{K})$ are the trivial supercharacter theories of $G$.

We now review some constructive results on supercharacter theories of finite groups. Hendrickson \cite{7} provided several constructions which are used to classify all supercharacter theories of cyclic groups and obtained an exact formula for the number of supercharacter theories of a finite cyclic $p$-group. By studying partitions of the set of irreducible characters,  Clifford theory and some well-known results regarding the structure of simple rational groups, Burkett et al.\cite{5} proved that  there are only three groups with exactly two supercharacter theories: the cyclic group $Z_3$, the symmetric group $S_3$ which is solvable, and the non-abelian simple group $Sp(6, 2)$.  Furthermore, Wynn\cite{14} described all supercharacter theories of extraspecial and Frobenius groups. The number of supercharacter theories of dihedral groups of order $2p$, $p$ is a Mersenne prime, was also calculated. In particular, he proved that if $G$ is a Frobenius group of order $pq$, where $p, q$ are primes and $p > q$, then $G$ has exactly $1 + \tau(\frac{p-1}{q})\tau(q-1)$ supercharacter theories in which $\tau(n)$ denotes the number of positive divisors of $n$. In \cite{1} the authors continued these works by providing some constructive methods in order to find new supercharacter theories. Then, they applied these methods to classify finite simple groups with exactly three or four supercharacter theories.

The aim of this paper is to present an algorithm for constructing all supercharacter theories of finite groups. To explain and then evaluate our algorithm, we need some concepts in computer science. The time complexity of a program with a given input data of size $n$ is defined as the number of elementary instructions that this program executes as a function of $n$. Moreover, the space complexity of a program with a given input data of size $n$ is defined as the number of elementary objects that this program needs to store during its execution with respect to $n$. Following Cormen et al.\cite{55}, we define:
\begin{center}
$\Theta(g(n))$ $=$ $\{f(n) \mid \ \exists \ c_1, c_2,  n_0 > 0 \ s.t. \ \forall ~n \geq n_0; 0 \leq c_1g(n) \leq f(n) \leq c_2g(n)\}.$ 
\end{center}
It can be seen that $f(n) \in \Theta(g(n))$ if there exist positive constants $c_1$ and $c_2$ in such a way that $c_1g(n)\leq f(n) \leq c_2g(n)$, for sufficiently large $n$. We use the notation $f(n) = \Theta(g(n))$ instead of $f(n) \in \Theta(g(n))$. Moreover, $f(n) \in O(g(n))$ if there are $c, n_0 > 0$ such that $f(n) \leq cg(n)$ whenever $n \geq n_0$, and $O(g(n))$ $=$ $\{f(n) \mid \exists ~c, n_0 >0; \forall ~n \geq n_0; 0 \leq f(n) \leq cg(n)\}.$ For two matrices $A$ and $B$ with the same number of rows, the augmented matrix $C= [A|B]$ is formed by appending the columns of $B$ to $A$.

Throughout this paper, our calculations are done with the aid of GAP\cite{10}. Our group theory notations and terminologies can be found in \cite{8,St}. Moreover, we refer the interested readers to the book\cite{55} for more information on algorithms.

\section{Algorithm}
Set $[n] = \{1, 2, \dots, n\}$, $G$ is a finite group, $Irr(G) = \{{\chi}_1,\ldots,{\chi}_n\}$ and $Con(G) = \{K_1,\ldots,K_n\}$. Choose $A \subseteq [n]$. Define $A_{I}(G)$ $=$  $\{{\chi}_i \mid i \in A\}$ and $A_{C}(G)$ $=$ $\{K_i \mid i \in A\}$. Then the mappings $\xi_1: \mathcal{P}([n]) \longrightarrow \mathcal{P}(Irr(G))$ and $\xi_2: \mathcal{P}([n]) \longrightarrow \mathcal{P}(Con(G))$ are given by $\xi_1(A) = A_I(G)$ and $\xi_2(A) = A_C(G)$, where $\mathcal{P}(Y)$ denotes the power set of a given set $Y$. Conversely, we assume that $B \subseteq Irr(G)$,  $C \subseteq Con(G)$ and define $[n]_B$ $=$ $\{i \mid \chi_i \in B\}$ and $[n]^C$ $=$ $\{j \mid K_j \in C\}$. We now define $\gamma_1 : \mathcal{P}(Irr(G)) \longrightarrow \mathcal{P}([n])$ and $\gamma_2 : \mathcal{P}(Con(G)) \longrightarrow \mathcal{P}([n])$ by $\gamma_1(B) = [n]_B$ and $\gamma_2(C) = [n]^C$. Then the mapping $\xi_t$ and $\gamma_t$, $t = 1, 2$, are mutually inverse and so we can use $\mathcal{P}([n])$ instead of both $\mathcal{P}(Con(G))$ and $\mathcal{P}(Irr(G))$ in our algorithms.  

Let  $X = \{\chi_1,\ldots, \chi_u\}$ and $K = \{K_1,\ldots, K_s\}$ be parts of set partitions $\mathcal{X}$ of $Irr(G)$ and $\mathcal{K}$ of $Con(G)$, respectively. If $\sigma_X(K_1) = \cdots = \sigma_X(K_s)$, then we say that $X$ and $K$ are \textit{consistent}. If all parts of $\mathcal{X}$ are mutually consistent with all parts of $\mathcal{K}$, then the set partitions $\mathcal{X}$ and $\mathcal{K}$ are said to be consistent. In a part of the proof of \cite[Theorem 2.2(c)]{6}, the following equivalence relation on $G$ is given:
$$u \sim v \Longleftrightarrow \forall ~X \in \mathcal{X},~ \sigma_X(u) = \sigma_X(v).$$

Note that if $u$ and $v$ are conjugate in $G$, then $u \sim v$. As a result, it is enough to compute $\sigma_{X}$ on all conjugacy classes of $G$. Suppose $I = \{X_1, \ldots, X_r\}$ is a set partition of $Irr(G)$ and define $\sigma_i = \sigma_{X_i}$, $1 \leq i \leq r$. Set $K_x = \{y \mid \forall i, 1 \leq i \leq r; \sigma_i(x) = \sigma_i(y)\}$, $x \in G$, and let $J$ be the set of all such subsets.
If all members of $J$ are non-empty, then $J$ is a set partition of $Con(G)$ consistent with $I$ and so $(I,J) \in Sup(G)$. It is not necessarily true that for each set partition $\mathcal{X}$ of irreducible characters there exists a set partition $\mathcal{K}$ of conjugacy classes such that $\mathcal{X}$ and $\mathcal{K}$ are consistent. Hence the problem of computing supercharacter theories of $G$ is reduced to the problem of computing all consistent pairs for $G$.

Now, we introduce some notations in order to work with supercharacter theories of a group $G$ in GAP. The notation 
$\sigma_X(i)$ denotes the image of $\sigma_X$ in the $i$-th conjugacy class of $G$. 

The aim of this section is to present an algorithm for constructing all supercharacter theories of a finite group $G$. We partition this algorithm  into three sub-algorithms. These sub-algorithms are presented in three sub-sections. In the first sub-section, a sub-algorithm for finding the set of all bad parts is provided. The second sub-section devotes to calculating all set partitions of an $n$-element set such that these set partitions do not have any bad part. In the last sub-section, a sub-algorithm for computing a consistent set partition of the conjugacy classes of a group $G$ with respect to a set partition of  $Irr(G)$ is given. 
In what follows, we provide their pseudocode in different sub-sections. 

\subsection{Bad Parts and Bad Set Partitions}
A part $X$ of a set partition $\mathcal{X}$ of $Irr(G)$ is said to be \textit{bad} if $X$ is consistent with only singleton subsets of $Con(G)$. A set partition containing a bad part is called a \textit{bad set partition}. It is easy to see that a bad set partition  $\mathcal{X}$ of $Irr(G)$ does not have a mate $\mathcal{K}$ such that $(\mathcal{X}, \mathcal{K}) \in Sup(G)$. In some cases like the cyclic groups of orders 11 and 13 and also the dihedral groups of orders $38$ and $46$, more than $\%96$ of all parts are bad. In such cases,  the running time decreases significantly by removing bad set partitions from the calculations.

We recall that in computing supercharacter theories of a finite group $G$, $\{e\}$ and $\{1_G\}$ are always parts of $\mathcal{K}$ and $\mathcal{X}$, respectively. As a result, it is enough to work with set partitions of $[n]^{\star}=\{2,\ldots, n\}$.

\begin{lem}\label{lem}
Suppose $X \in \mathcal{P}([n]^{\star})$ is a bad part. Then all values of $\sigma_{X}(i)$, $2 \leq i \leq n$, are distinct.
\end{lem}

\begin{proof}
Choose $2 \leq j \neq k \leq n$ such that $\sigma_{X}(j) = \sigma_{X}(k)$. Then the part $X$ is consistent with $\{j\}$, $\{k\}$ and $\{j, k\}$. This is a contradiction to the definition of a bad part.
\end{proof}

By Lemma \ref{lem}, to check whether a part $X \in \mathcal{P}([n]^{\star})$ is bad or not, it is enough to compute $\sigma_X(i)$ for $i \in [n]^{\star}$. If all values are different, then $X$ is a bad part. The list of all bad parts can be computed by the following pseudocode. We use the command ``FindBadParts(G)" to call it.\\

\begin{tabular}{l}
\hline
\textbf{Sub-Algorithm 1} Find Bad Parts \\
\hline
\hspace*{.3cm}\textbf{Input:} A given group G \hspace*{7cm} \\
\hspace*{.3cm}\textbf{Output:} BadParts, list of all bad parts of G\\
\hspace*{.3cm}FindBadParts(G)\\
\hspace*{1.1cm} t  := Sorted character table of G\\
\hspace*{1.1cm} n := The number of conjugacy classes of G\\
\hspace*{1.1cm} BadParts :=[ ]\\
\hspace*{1.1cm}$AllParts := \mathcal{P}([n]^{\star})$\\
\hspace*{1.1cm}\textbf{for} each part in AllParts \textbf{do}\\
\hspace*{1.6cm}R := An empty set \\
\hspace*{1.6cm}\textbf{for} each c in $[n]^{\star}$ \textbf{do}\\
\hspace*{2.3cm}$R := R \cup \sigma_{part}(c)$\\
\hspace*{1.6cm}\textbf{end for}\\
\hspace*{1.6cm}\textbf{if} $|R| = n-1$ \textbf{then}\\
\hspace*{2.3cm}$BadParts := BadParts \cup \{part\}$\\
\hspace*{1.6cm}\textbf{end if}\\
\hspace*{1.1cm}\textbf{end for}\\
\hspace*{.9cm} \textbf{return} BadParts\\
\hspace*{.3cm}\textbf{end}\\
\hline
\end{tabular}

\subsection{Create Set Partitions}
A simple calculation by GAP shows that there are $82864869804$ set partitions for the case $n = 17$. When we run the command ${\rm PartitionsSet([1..17])}$, the following message appears: Error, reached the pre-set memory limit.  The GAP command
${\rm PartitionsSet([1..n])}$ generates all set partitions of $[n]$ and save them on the memory of the computer. Here, it is useful to mention that GAP has another command ${\rm IteratorOfCombinations(}$ $\rm{[1..n], i)}$ that does not need to store all elements of the collection under investigation. Since this command is time consuming, it is not efficient enough for our purpose. To solve this problem, we design a new algorithm for generating set partitions without saving them on RAM.

In literature, there are two algorithms by Semba\cite{semba} and Er\cite{er} for computing set partitions of $[n]$. 
The Semba's algorithm which is based on the backtrack technique
\cite[Theorem 1]{semba} has the time complexity of $\Theta(4B(n))$, where the Bell number $B(n)$ is defined as the number of set partitions of an $n$-element set. The Er's algorithm is recursive. He claimed (without proof) that $\sum_{i=1}^nB(i) < 1.6B(n)$. This is while Nayak and Stojmenovi\'{c}\cite[p. 12]{99} proved that $\sum_{i=2}^nB(i) < 2B(n)$.  In an exact phrase, Er claimed that the time complexity of his  algorithm for generating all set partitions of $[n]$ is $\Theta(1.6B(n))$.

We now explain the Er's algorithm. Choose a set partition $P = \{\pi_1, \pi_2, \ldots, \pi_k\}$ of $[n]$. Define the codeword $c(P) = c_1c_2\ldots c_n$ such that $1 \leq c_i \leq i$ and $c_i = j$ if and only if $i \in \pi_j$. It is easy to see that there exists a one-to-one correspondence between the set of all set partitions and the set of all such codewords. In Er's algorithm, codewords are computed with the given property as the set partitions of $[n]$. 

There exists a limitation in the Er's algorithm: All codewords are determined at the end step and so we cannot identify whether a given set partition is bad or not, before it is done completely. As a consequence, we design an algorithm which generates set partitions part by part. When a bad part occurs, calculations of all set partitions containing that part are pruned. 

To explain our algorithm, we set $S = \{ s_1, s_2, \ldots, s_n\}$. The $2^n - 1$ non-empty subsets of $S$ are used as the parts of the set partitions of $S$. Note that when $n$ is large enough, it is not possible to save all set partitions on the memory. In order to save the memory, we use the integers of the closed interval $I = I(S) = [1,2^{|S|}-1]$. 
 In fact, we define a one-to-one correspondence $\alpha_S: I \longrightarrow \mathcal{P}(S) \setminus \{\emptyset\}$  by $\alpha_S(k) = \{ s_i \mid a_{n+1-i} = 1\}$, where  $a_1a_2\ldots a_n$ is the $n$-bit binary form of $k$. For example, if $S = \{ s_1, s_2, s_3, s_4, s_5\}$ and $k=13$ then  $(13)_2 = 1101$ and since $|S| = 5$, the $5$-bit binary form of $13$ is $01101$. Thus, $\alpha_S(13) = \{ s_1, s_3, s_4\}$. Let $I_O = I_O(S)$ be the set of all odd integers in the closed interval $I$. Then $\alpha_S(I_O)$ is the set of all subsets of $S$ containing $s_1$. On the other hand, if $F$ is a non-empty subset of $S$, then we conclude that $\alpha_S^{-1}(F)$ $=$ $\sum_{i \in F}2^{Position(S,i) - 1}$,  where   $Position(S,i)$, $i \in F$, is the position of $i$ in $S$. In this example, if $F = \{ s_1, s_3, s_4\}$ then $\alpha_S^{-1}(F)$ $=$ $2^{1 - 1} + 2^{3 - 1} + 2^{4 - 1} = 13$.

Now, we are ready to present our algorithm for constructing all set partitions with no bad part. We use two lists $SPs$ and $RE$ in order to keep set partitions and remaining elements, respectively. At the first step, we have $SPs = \emptyset$ and $RE = [n]^{\star}$. We fill $SPs$ by parts constructed from the elements of $RE$ such that these parts are not bad. In other words,
$RE := RE \setminus \alpha_{RE}(k)$ and $SPs := SPs \cup \{ \alpha_{RE}(k) \}$, where  $k \in I_O$ and $\sigma_{RE}(k)$ is not a bad part. This algorithm will be returned to the previous step, when $RE = \emptyset$. Since our algorithm is recursive, it's generating tree is constructed by DFS strategy.

In the Sub-algorithm 2, we present a pseudocode for a part of our main algorithm. In this sub-algorithm which is called by the command CreateSetPartitions(RE, BP), all set partitions of $[n]^{\star}$ that do not  contain any bad part are generated. If we remove the condition $\sigma_{RE}(k) \notin BP$ from the algorithm and replace the command ``CreateKappa(SPs)" by ``Print(SPs)", then the new algorithm can generate all set partitions with no pruning.

\vskip 3mm

\begin{tabular}{l}
\hline
\textbf{Sub-Algorithm 2} Create Set Partitions Based on Filtering Bad Parts\\
\hline
\hspace*{.3cm}\textbf{Input:} RE, A set of numbers; BP, the list of bad parts \hspace*{2cm}\\
\hspace*{.3cm}\textbf{Output:} All set partitions of RE without any bad part\\
\hspace*{.3cm}SPs := A set to keep each set patition; at first it equals to the empty set\\
\hspace*{.3cm}CreateSetPartitions(RE, BP)\\
\hspace*{.5cm}\textbf{if} RE = $\emptyset$ \textbf{then}\\
\hspace*{1cm}CreateKappa(SPs)\\
\hspace*{.5cm}\textbf{else}\\
\hspace*{1cm}\textbf{for} each $k \in I_O(RE)$ \textbf{do}\\
\hspace*{1.7cm} \textbf{if} $\sigma_{RE}(k) \notin BP$ \textbf{then}\\
\hspace*{2.3cm}$newRE := RE \setminus \alpha_{RE}(k)$\\
\hspace*{2.3cm}$SPs := SPs \cup \{\alpha_{RE}(k)\}$\\
\hspace*{2.3cm}CreateSetPartitions(newRE, BP)\\
\hspace*{2.3cm}$SPs := SPs \setminus \{\alpha_{RE}(k)\}$\\
\hspace*{1.7cm}\textbf{end if}\\
\hspace*{1cm}\textbf{end for}\\
\hspace*{.5cm}\textbf{ end if}\\
\hspace*{.3cm}\textbf{end}\\
\hline
\end{tabular}

\vskip 3mm

In Figure \ref{re4}, an example of a generating tree for set partitions of $[4]^{\star}$ is presented. 	

\begin{figure}[htp]
\centering
\includegraphics[scale=.75]{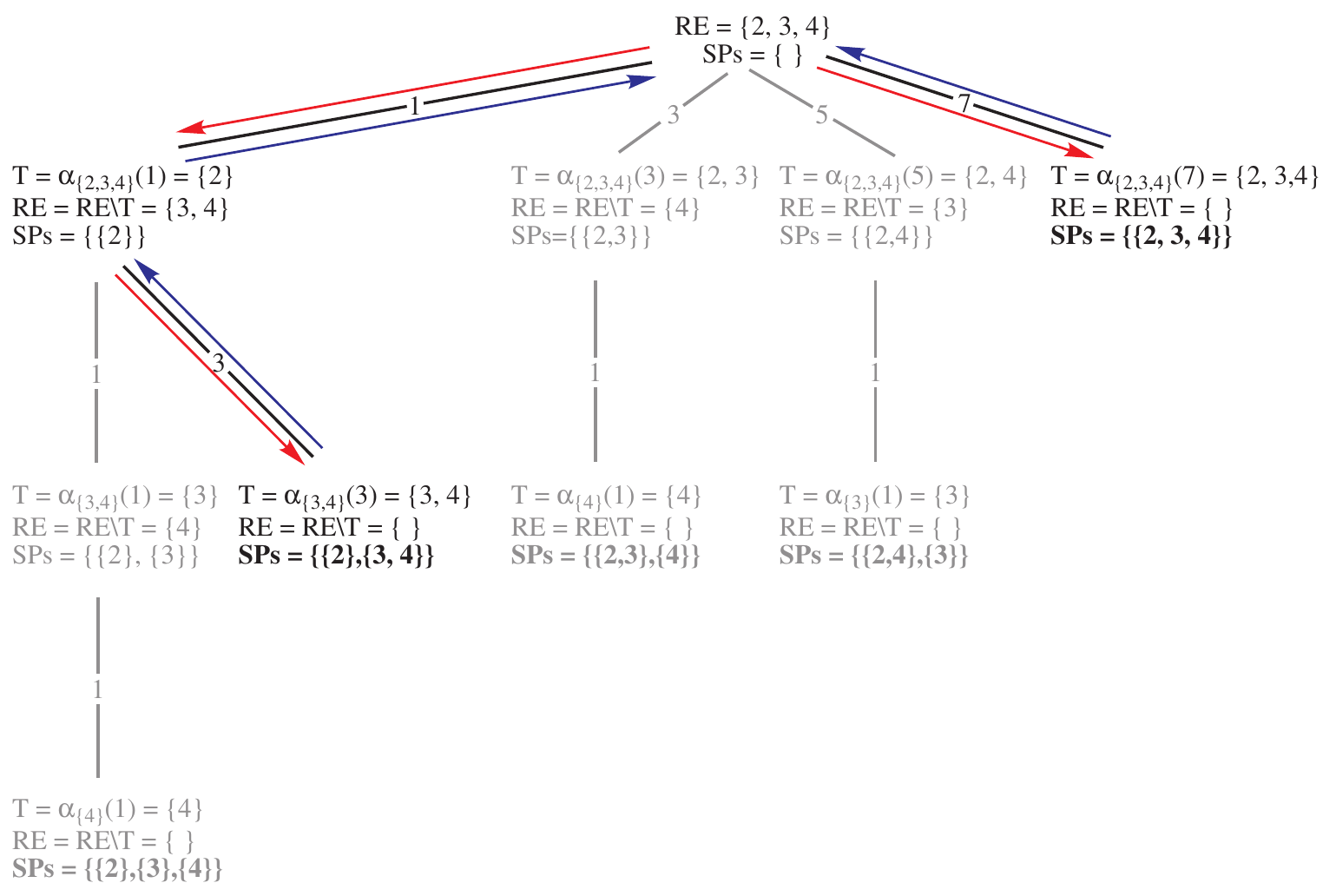}
\caption{A schematic diagram for the case ${\rm CreateSetPartitions([2..4],}$ ${\rm \{\{2, 3\}, \{2, 4\}, \{3\}, \{4\}\})}.$ The red and blue arrows represent forward and backward directions in the generating tree traversal, respectively. The number on each edge is an odd integer in $I_O(RE)$ of the parent node of the edge. The gray part of the tree is pruned due to the occurrence of a bad part in the process of creating corresponding set partition. }\label{re4}
\end{figure}

Note that after creating a set partition for $Irr(G)$, we invoke the ``CreateKappa" sub-algorithm for it. This is to check whether there exists a consistent set partition of $Con(G)$ or not. This function is explained in details in the next subsection.

\newpage

\subsection{Create the Consistent Set Partition $\mathcal{K}$ with respect to a Given Set Partition of $Irr(G)$}
We recall that finding all supercharacter theories of a group $G$ with $n$ conjugacy classes is equivalent to constructing all consistent pairs of set partitions. Suppose $Irrp$ is a given set partition of the irreducible characters of $G$. To find a consistent set partition of $Irrp$, we first define the matrix $A$ as follows (see Table \ref{tg2}).
\begin{itemize}
\item The rows of $A$ are the parts of $Irrp$ and so $A$ has exactly $|Irrp|$ rows.
\item The columns of $A$ are the conjugacy classes of $G$.
\item If $A = (a_{ij})$, then $a_{ij} = \sigma_{X_i}(K_j)$, where $1 \leq i \leq |Irrp|$, $1 \leq j \leq n$ and $K_j$ is a conjugacy classes of $G$.
\end{itemize}

Suppose $C_1$, $C_2$, $\ldots$, $C_n$ are all columns of $A$. We construct a matrix $ST$ and a list $Kappa$ as follows. Since $\{e\}$ is a part of each consistent set partition with $Irrp$,  we conclude that $\{1\} \in Kappa$. We start our algorithm by defining $Kappa = \{\{1\}, \{2\}\}$ and the submatrix $ST = [C_1,C_2]$.  For each $j$, $3 \leq j \leq n$, we compare $C_j$ with all constructed columns of $ST$ other than its first column. If $C_j$ is different from such columns of $ST$, then we add $C_j$ to $ST$ as a new column and add $j$ to $Kappa$ as a singleton part. Hence $Kappa := Kappa \cup \{\{j\}\}$ and $ST := [ST|C_j]$. If $C_j$ is equal to the $r$-th column of $ST$, then we add $j$ to the $r$-th part of $Kappa$, i.e.
$$Kappa = \{\{1\}, \ldots, \{\underbrace{\ldots, j}_{part~ r}\}, \ldots ,\{\ldots \}\}.$$

\begin{table}[htp]
\centering\caption{Matrix $A$.}\label{tg2}
\begin{tabular}{c|ccccc}
&$K_1$&$K_2$&$K_3$&$\cdots$ & $K_n$ \\
\hline
$\chi_1$&1 &1 &1 & $\cdots$ & 1\\
\hline
$\vdots$  &* & *& * &$\cdots $ &*\\
\hline
$X_i \left\{ \begin{array}{l}
\chi_{i_1}\\
\vdots\\
\chi_{i_t}
\end{array}\right.$ & $\sigma_{X_i}(K_1)$ & $\sigma_{X_i}(K_2)$  & $\sigma_{X_i}(K_3)$ & $\cdots$ & $\sigma_{X_i}(K_n)$ \\
\hline
$\vdots$  &* & *& * &$\cdots $ &*\\
\end{tabular}
\end{table}

If in the process of constructing $Kappa$ and $ST$ the inequality $|Kappa|>|Irrp|$ occurs, then we stop calculations without any result. It is because there is no consistent set partition with the same size as $Irrp$. If at the end of our calculations, $|Kappa| = |Irrp|$, then we conclude that $(Irrp, Kappa)$ is a supercharacter theory and $ST$ is the supercharacter table of $G$.

In Sub-algorithm 3, our pseudocode for computing a supercharacter theory of a group $G$ is presented. The input of the program $CreateKappa$ is a set partition of the irreducible characters of a group $G$  named $Irrp$ and its output is the consistent set partition of $Con(G)$ with respect to $Irrp$.

\newpage

\begin{tabular}{l}
\hline
\textbf{Sub-Algorithm 3} Create Kappa for a Given Set Partition of Irr(G)\\
\hline
\hspace*{.3cm}\textbf{Input:} A given group G, Irrp, A set partition of Irr(G) \hspace*{2cm} \\
\hspace*{.3cm}\textbf{Output:} (Irrp, Kappa), A supercharacter theory of G for a given Irrp \\
\hspace*{4cm}(if exists) \\
\hspace*{.3cm}CreateKappa(Irrp)\\
\hspace*{1.1cm}t := Character table of G\\
\hspace*{1.1cm}n := $|Con(G)|$\\
\hspace*{1.1cm}\textbf{for} each $part_i$ in Irrp \textbf{do}\\
\hspace*{1.7cm}\textbf{for} each j in $\{1,2, \ldots, n\}$ \textbf{do}\\
\hspace*{2.3cm}$A[i][j] := \sigma_{part_i}(K_j)$\\
\hspace*{1.7cm}\textbf{end for}\\
\hspace*{1.1cm}\textbf{end for}\\
\hspace*{1.1cm}ST := [C1, C2] //Ci is the i-th column of A\\
\hspace*{1.1cm}Kappa := [[1],[2]]\\
\hspace*{1.1cm}\textbf{for} each $C_j$ ($j \geq 3$) of A \textbf{do}\\
\hspace*{1.7cm}compare $C_j$ to all columns of ST except the first one\\
\hspace*{1.7cm}\textbf{if} $C_j$ equals to the r-th column of $ST$ \textbf{then}\\
\hspace*{2.3cm}$Kappa[r] :=Kappa[r] \cup \{j\}$\\
\hspace*{2.3cm}\textbf{if} $|Kappa|>|Irrp|$ \textbf{then}\\
\hspace*{2.9cm}\textbf{return}\\
\hspace*{2.3cm}\textbf{end if}\\
\hspace*{1.7cm}\textbf{else} \\
\hspace*{2.3cm}$ST := [ST | Cj]$\\
\hspace*{2.3cm}$Kappa := Kappa \cup \{[j]\} $\\
\hspace*{1.7cm}\textbf{end if}\\
\hspace*{1.1cm}\textbf{end for}\\
\hspace*{1.1cm}\textbf{if} $|Kappa| = |Irrp|$ \textbf{then}\\
\hspace*{1.7cm}print(Irrp, Kappa)\\
\hspace*{1.1cm}\textbf{end if}\\
\hspace*{.3cm}\textbf{end}\\
\hline
\end{tabular}

\vskip 3mm

To construct all supercharacter theories of a group $G$, it is possible to combine the Er's algorithm which creates all set partitions of $Irr(G)$ with the algorithm based on the proof of Theorem 2.2(c) in \cite{6}. This is our \textit{first algorithm}. Our main algorithm is a combination of the Sub-algorithms 1, 2 and 3. The pseudocode of the main algorithm is as follows.

\begin{flushleft}
\begin{tabular}{l}
\hline
\textbf{Main Algorithm} Find all Supercharacter Theories of $G$ by Filtering Bad Parts \\
\hline
\hspace*{.3cm}\textbf{Input:} A given group G  \\
\hspace*{.3cm}\textbf{Output:} All pairs (Irrp, Kappa) as supercharacter theories of G\\
\hspace*{.3cm}FindSupercharacterTheories(G)\\
\hspace*{1.5cm}BP := FindBadParts(G)\\
\hspace*{1.5cm}n := $|Con(G)|$\\
\hspace*{1.5cm}${\rm CreateSetPartitions([n]^{\star}, BP)}$\\
\hspace*{.3cm}\textbf{end}\\
\hline
\end{tabular}
\end{flushleft}

In the following example, our sub-algorithm for computing bad parts of the cyclic group $Z_{13}$ is analyzed. The notation $SmallGroup(n, i)$ stands for the $i$-th group of order $n$ in the small group library of GAP.

\begin{example}\label{ex1}
Suppose $G = Z_{13}$. Then $G$ has exactly $4095$ non-empty subsets which can be a part of a set partition of $Irr(G)$. Our calculations with GAP show that among these subsets, there are $4020$ bad parts. The set partitions containing at least one of these bad parts have to be deleted. 
Note that $\{2\}$ and $\{3\}$ are bad parts. There are $B(10) =115975$ set partitions which contain $\{2\}$ or $\{3\}$ as a part and so they have to be deleted from our investigations. Consequently,  $96.18\%$ of all parts of  $\mathcal{P}(Irr(Z_{13})\setminus \{1_{Z_{13}}\}) = \mathcal{P}([13]^{\star})$ are bad parts. If for partitioning $Irr(Z_{13})$ we apply the Er's algorithm, then the program for computing all supercharacter theories takes so long to run. Note that the Er's algorithm does not have this potential to find bad parts. 
For example, our algorithm that is presented in this paper takes less than one second for computing all supercharacter theories of $Z_{13}$, while for the Er's algorithm we need almost $548$ seconds.
\end{example}

We end this section by noticing that:
\begin{enumerate}
\item The result of our main algorithm is supercharacter theories of a given group $G$. In fact, conditions of being a supercharacter theory are checked by the main algorithm in each case.

\item All supercharacter theories are generated by our algorithm. This is guaranteed by the proof of \cite[Theorem 2.2(c)]{6}. \\

\end{enumerate}

\section{Analysis of Algorithms}
In Section $2$, three sub-algorithms for computing bad parts, set partitions and supercharacter theories were presented. The aim of this section is to calculate the running time and the space complexity of these sub-algorithms and our main algorithm.

\begin{thm}
Let $T_1(n)$, $S_1(n)$ and $BP$ be the time complexity function, the space complexity function and the list of bad parts for a given group, respectively. Then, $T_1(n) \in O((n^2-n) \cdot 2^{n-1})$ and $S_1(n) \in O(n) + O(|BP|)$, where $n=|Con(G)|$.
\end{thm}
\begin{proof}
To compute the running time of the function $FindBadParts$, we should know values of $\sigma_X(j)$, $2 \leq j \leq n$, where $X = \{x_1,\ldots, x_i\} \in \mathcal{P}([n]^{\star})$. For this purpose, $i(n-1)$ multiplications and  $(i-1)(n-1)$ additions are needed. Then, we have $\frac{(n-1)(n-2)}{2}$ comparisons for investigating the property that $\sigma_X$s are distinct. Since there are ${n-1 \choose i}$ $i$-subsets, the complexity of this sub-algorithm can be computed by the following formula:
$$T_1(n) = \sum_{i=1}^{n-1}\left[ (n-1)(2i-1){n-1 \choose i} + \frac{(n-1)(n-2)}{2}\right].$$
Therefore,

\begin{eqnarray*}
T_1(n) &=& \sum_{i=1}^{n-1}\left[ (n-1)(2i-1){n-1 \choose i} + \frac{(n-1)(n-2)}{2}\right]\\
      &=& \left( n-1\right)  \sum_{i=1}^{n-1}(2i-1){n-1 \choose i} + \frac{(n-1)^2(n-2)}{2} \\
       &<& n \cdot  \sum_{i=1}^{n-1}2i \cdot {n-1 \choose i} + n^3\\
       &=& 2 n\cdot \sum_{i=1}^{n-1}i \cdot {n-1 \choose i} + n^3\\
       &=& 2n\cdot (n-1) \cdot 2^{n-2} + n^3 = (n^2-n) \cdot 2^{n-1} + n^3.
\end{eqnarray*}
Hence $T_1(n) \in O((n^2-n) \cdot 2^{n-1})$. 

To compute the space complexity of this sub-algorithm, we note that all parts are generated one by one. If a generated part is bad, then we add it to $BP$. To keep each part, our calculations need an array of size $n-1$. Moreover, an $(n-1)$-length array is needed in order to save the values $\sigma_X(i)$. Therefore, this sub-algorithm needs a memory of size $O(n)$ to keep each part. For saving all bad parts, we need another array such that its size depends only on the number of bad parts of irreducible characters of a given group. Consequently, $S_1(n) \in O(n) + O(|BP|)$.
\end{proof}

\begin{thm}\label{th1}
Suppose $T_2(n)$  and $S_2(n)$ are the time and space complexity functions of $Create$- $SetPartitions(RE, BP)$, respectively. Then,
\begin{enumerate}
\item $T_2(n) \in O(2B(n))$;
\item The space complexity $S_2(n)$ belongs to $\Theta(n)$.
\end{enumerate}
\end{thm}

\begin{proof}
To prove (1), let $T_2(n)$ denote the number of calculations needed to obtain all set partitions without any bad part of the $(n-1)$-element set $RE$. For computing the time complexity in the worst case $|BP|=0$, we have to count the number of edges in the generating tree of the function ${\rm CreateSetPartitions(RE, BP)}$ in general, see Figure \ref{re4}. Then,
$$T_2(n) = \displaystyle\sum_{i=0}^{n-1}{n-1 \choose i}(T_2(i)+1).$$

In OEIS \cite{oeis}, the sequence $\{a(n)\}_{n \geq 0}$ with code $A060719$ exists which is defined as follows.
$$ a(n+1)=a(n)+\sum_{i=0}^{n}{n \choose i}(a(i)+1); \ a(0) =1$$
By \cite{999}, $a(n) = 2B(n+1) - 1$ and since $T_2(n) = a(n-1)$, we conclude that $T_2(n) = 2B(n) - 1$. Hence, the time complexity of this algorithm is $O(2B(n))$. 

The space complexity depends on the sizes of $SPs$ and $RE$. Since the union of $RE$ with the members of $SPs$ is $[n]$, $S_2(n) \in \Theta(n)$ which proves (2).
\end{proof}
In the next theorem, we calculate the complexity of our Sub-algorithm 3 which computes the supercharacter theories of $G$.
\begin{thm}\label{th2}
Suppose $G$ has exactly $n$ conjugacy classes. For a given set partition $Irrp$, the time and space complexities of the function $CreateKappa$  are $O(n^3)$ and $O(n^2)$, respectively.
\end{thm}
\begin{proof}
Suppose $Irrp$ is a set partition for the set of all irreducible characters of a group $G$ and  $|Irrp| = k$. To calculate the matrix $A$, we first obtain all values $\chi_i(1)\chi_i$, $1 \leq i \leq n$. Since $\chi_i$ has $n$ values, there are $n^2$ different products   $\chi_i(1)\chi_i$. To compute $\sigma_X$ and in the worst case $X = \{\chi_1, \ldots, \chi_n\}$, we need $n(n-2)$ products. As a result, the calculations for obtaining the matrix $A$ is of the time complexity $O(n^2)$.

Now, we count all the operations that we need to construct $ST$ and the list $Kappa$. The first and second columns of $ST$ are the same as the first and second columns of $A$, respectively. Therefore, we have nothing to count for these columns.
 For the third column of $ST$, we have to compare the third column of $A$ with the second column of $ST$ and so there are $k$ comparisons. For the fourth column of $ST$, the fourth column of $A$ should be compared with the second and third columns of $ST$, and so, there are at most $2k$ comparisons, and so on. Suppose that from the column $C_j$ to the next, $|Kappa| = |Irrp|$. In this case, the remaining conjugacy classes of the group should be distributed among the other parts of $Kappa$ and hence we do not have a new part in $Kappa$. Thus from $C_j$ to $C_n$, the number of comparisons is equal to $k(k-1)$. Therefore, the total number of comparisons for constructing $ST$ and $Kappa$ is:
$$T_3(n) := k(1) + k(2) + \cdots +\underbrace{k(k-1) + k(k-1) + \cdots + k(k-1)}_{n-j+1}.$$ 
Since $j \geq 3$, $n-j+1 \leq n-2$. Thus $T_3(n) \leq k(\frac{n^2-5n+6}{2}) \leq n(\frac{n^2-5n+6}{2})$, and so, $T_3(n) \in O(n^3)$.

Suppose $S_3(n)$ is the space complexity of the function $CreateKappa$. We have two matrices $A$ and $ST$ of sizes $k \times n$ and $k \times k$, respectively. Since in the worst case $k = n$, we conclude that $S_3(n) \in O(n^2)$.
\end{proof}

We are now ready to compute the time and space complexity of the first and main algorithms. In the first algorithm, the generated set partitions are used as an input to compute all supercharacter theories of a finite group. In what follows, we assume that our group has exactly $n$ conjugacy classes. The running time of our first algorithm is $T_{Er}(n) = O(n^3).\Theta(2B(n))$ and so $T_{Er}(n) \in O(n^3B(n))$. 

In the main algorithm, we do not need to call the sub-algorithm CreateKappa for bad set partitions. Therefore, the sub-algorithm CreateKappa should be called $B(n)-|BP_s|$ times in order to calculate $Kappa$, where $|BP_s|$ is the number of bad set partitions. As a result, the time complexity of the main  algorithm is
\begin{center}
 $O((n^2-n)2^{n-1}+(B(n)-|BPs|)n^3) = O(n^3(B(n)-|BPs|))$.
\end{center} 
 Consequently, we have the following result:
\begin{theorem}
The time complexity of our first and main algorithms are $O(n^3B(n))$ and $O(n^3(B(n)-|BPs|))$, respectively.
\end{theorem}



\section{Performance Evaluation}
To evaluate the performance of the main algorithm and then compare it with the first one, both algorithms have been implemented in the computer algebra system GAP under Windows $10$ Home Single Language. 
The average running times for both algorithms after three runs on a computer with processor Intel(R) Core(TM) m7-6Y75 CPU @ 1.20 GHz 1.51 GHz, installed memory (RAM) 8.00 GB (7.90 GB usable), system type $64$-bit operating system and x64-based processor are summarized in Table \ref{tt3}. In this table, we have chosen groups which have the maximum or the minimum number of supercharacter theories with different number of conjugacy classes. Let $BP(G)$ be the set of all bad parts in a group $G$ and $\alpha(G) = \frac{|BP(G)|}{2^{\kappa(G) - 1}-1}$ in which $\kappa(G)$ denotes the number of distinct conjugacy classes of $G$. We have the following two cases in general.
\begin{enumerate}
\item \textbf{There is not any bad part in $\mathcal{P}(Irr(G))$.} In this case, the algorithm for computing supercharacter theories based on the Er's algorithm have a faster running time. Note that we have a pre-process for finding bad parts but such an overhead is very small with respect to the total running time.

\item \textbf{There are some bad parts in $\mathcal{P}(Irr(G))$.} In this case, by removing  these parts  from our calculations, the main algorithm will have a faster running time. For example, in the cyclic group $Z_{13}$ in which $\%98.17$ of all parts are bad, our main algorithm takes less than one second to run while the other algorithm takes more than 548 seconds. In rare cases such as the Mathieu group $M_{22}$ in which a few percentage of parts are bad, the running time of the first algorithm is a bit faster.
\end{enumerate}

\begin{table}[htp]
\centering
\caption{Comparing the running times for some groups.}\label{tt3}
\begin{tiny}
\begin{tabular}{|c|c|c|c|c|c|c|c|} \hline
$\kappa(G)$ & G & $|Sup(G)|$  & $|BP(G)|$& $\alpha(G)$ & Main algorithm & First algorithm & FA/MA\\
 &  &  &  &   & (MA)(second)& (FA)(second) &\\ \hline
10 & $[100,11]$ & 623  & 0 & 0&1.4 & 1.1  & 0.8\\ \hline
11 & $[32,43]$ & 376  & 0 &  0&7.6 & 6.7 & 0.9\\ \hline
11 & $[32,44]$ & 376  & 0 &  0& 8.1 & 7.5 & 0.9\\ \hline
12 & $[1296,3523]$ & 1058  & 0 & 0& 70.1 & 62 &0.9\\ \hline
13 & $[64, 32]$  & 325  & 0 &  0&464.6 & 429.8 & 0.9\\ \hline \hline

12 & D36 & 51  & 168 & 8.2 & 65.4 & 68.2 & 1.04\\ \hline
12 & M22 & 5  & 288 & 14.1 & 65.8 & 61.8 & 0.9\\ \hline
10 & M11 & 5  & 112 & 21.9 & 0.8 & 1.1 & 1.4\\ \hline
10 & D28 & 23  & 144 & 28.9 & 0.8 & 1.7 & 2.1\\ \hline
10 & $[120,35]$ & 10  & 152 & 29.7 & 0.6 & 1.1 &  1.8\\ \hline
13 & $[93,1]$  & 9  & 1980 & 48.4 & 169.7 & 662.3 & 3.9\\ \hline
13 & $[253, 1]$ & 9  & 1980 & 48.4 & 127.8 & 521.3 &4.08\\ \hline
10 & Z10 & 10  & 376 & 73.6 &  0.06 & 1.2 & 20\\ \hline
10 & D34 & 5  & 480 & 93.9 &  0.04 & 1.3 &  32.5\\ \hline
11 & Z11 & 4  & 990 & 96.8 &  0.1 & 7.9 & 79\\ \hline
13 & Z13 & 6  & 4020 & 98.1 &  1.2 & 548.2 & 456.8\\ \hline
11 & D38 & 4  & 1008 & 98.5 & 0.1& 8.5 & 85\\ \hline
13 & D46 & 3  & 4092 & 99.9 & 1.2 & 610.6 & 508.8 \\ \hline
\end{tabular}
\end{tiny}
\end{table}

To compare the first and main algorithms, the running time of the groups with exactly $13$ conjugacy classes with respect to these algorithms are depicted in Figure \ref{fig3}. In this figure, groups numbered 29-53 have faster running times with the main algorithm. The result shows that the main algorithm is better for some classes of groups.

\begin{figure}[htp]
\centering
\includegraphics[scale=.4]{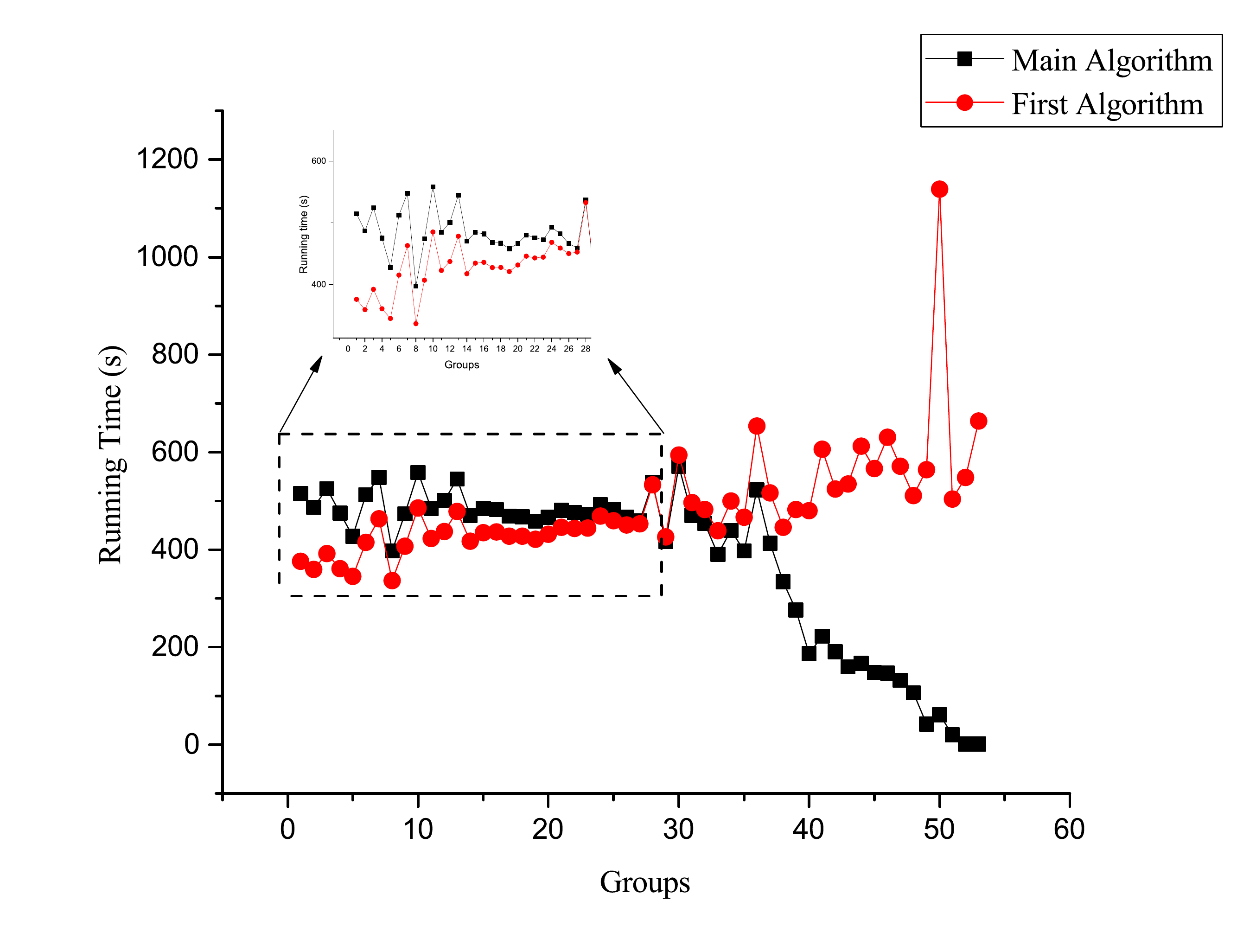}

\caption{A diagram for the running time of all $53$ groups which are listed in Table \ref{t5}.} \label{fig3}

\end{figure}

\begin{table}[htp]
\centering
\caption{The GAP id of all groups with exactly 13 conjugacy classes.}\label{t5}
\begin{tabular}{|c|l||c|l||c|l||c|l||c|l|}
\hline
$1 $ &  $[162,21]$ & $ 12$ & $[162,20] $  & $23 $ & $[960,11359] $  & $ 34$ & $[100,10] $  & $45 $ & $[328,12] $ \\
\hline
$2 $ & $[96,191] $ & $13 $ & $[1944,2290] $ & $24 $ & $[64,33] $  & $ 35$ & $[162,15] $  & $46 $ & $[148,3] $ \\
\hline
$3 $ & $[400,206] $ & $14 $ & $[64,37] $ & $25 $ & $[1000,86] $ & $ 36$ & $[40,6] $  & $47 $ & $[333,3] $  \\
\hline
$4 $ & $[162,22] $ & $15 $ & $[64,32] $ & $26 $ & $[720,409] $ & $ 37$ & $[1053,51] $  & $48 $ & $[156,7] $ \\
\hline
$5 $ & $[192,1494] $ & $16 $ & $[96,190] $ & $27 $ & $[576,8652] $ & $ 38$ & $[120,38] $  & $49 $ & $[301,1] $  \\
\hline
$6 $ & $[96,193] $ & $17 $ & $[192,1491] $ & $28 $ & $[40,4] $ & $ 39$ & $[600,148] $  & $50 $ & $[205,1] $ \\
\hline
$7 $ & $[1944,2289] $ & $18 $ & $[64,36] $ & $29 $ & $\textbf{[162,11]} $ & $ 40$ & $[216,86] $  & $51 $ & $[150,5] $ \\
\hline
$8 $ & $[192,1493] $ & $19 $ & $[192,1492] $ & $30 $ & $[160,199] $ & $ 41$ & $[258,1] $  & $52 $ & $[13,1] $ \\
\hline
$9 $ & $[1440,5841] $ & $20 $ & $[64,35] $ & $31 $ & $[324,160] $ & $ 42$ & $[310,1] $  & $53 $ & $[46,1] $ \\
\hline
$10 $ & $[40,8] $ & $21 $ & $[162,19] $ & $32 $ & $[162,13] $ & $ 43$ & $[253,1] $  &  &  \\
\hline
$11 $ & $[64,34] $ & $22 $ & $[216,87] $ & $33 $ & $[1320,133] $ & $ 44$ & $[93,1] $  &  & \\
\hline
\end{tabular}
\end{table}

\section{Concluding Remarks}
In this paper, two algorithms for computing all supercharacter theories of a finite group $G$ have been presented. The first algorithm is based on the Er's algorithm. In the main algorithm, we have introduced the new feature ``bad part" for the parts of $Irr(G)$. 
Since none of the supercharacter theories contains these bad parts, by filtering and detecting the set partitions of $Irr(G)$ which have at least one bad part, the running time of this algorithm decreases significantly. 

Suppose $BP(G)$ denotes the set of all bad parts in a group $G$,  $\alpha(G) = \frac{|BP(G)|}{2^{\kappa(G) - 1}-1}$ and $|Sup(G)|$ is the number of supercharacter theories of $G$. In Table \ref{3}, the percentage of bad parts for some cyclic and dihedral groups are given.

\begin{table}[htp]
\centering
\caption{Percentage of bad parts for some cyclic and dihedral groups.}\label{3}
\begin{tabular}{|c|c|c|c||c|c|c|c|} \hline
$ \kappa(G)$ &Group & $\alpha(G)$ & $|Sup(G)|$ &$ \kappa(G)$ & Group & $\alpha(G)$ & $|Sup(G)|$\\
\hline
4& $D_4$ &$\% 0$ &5 & 2 & $Z_2$ &$\% 100$ & 1  \\ \hline
3& $D_6$ &$\% 66.67$ &2  & 3 & $Z_3$ &$\% 66.67$ & 2\\ \hline
4& $D_{10}$ &$\% 57.14$ &3 & 5 & $Z_5$ &$\% 80$ &3 \\ \hline
5& $D_{14}$ &$\% 80$ & 3 & 7 & $Z_7$ &$\% 85.7$ &4 \\ \hline
7& $D_{22}$ &$\% 95$ &3 & 11 & $Z_{11}$ &$\% 96.77$ &4  \\ \hline
8& $D_{26}$ & $\% 84.3$&5 & 13 & $Z_{13}$ &$\% 98.16$ &6  \\ \hline
10& $D_{34}$ &$\% 93.75$ &5 & 17 & $Z_{17}$ &$\% 99.6$ &5  \\ \hline
11& $D_{38}$ &$\% 98.53$ &4 & 19 & $Z_{19}$ &$\% 99.78$ &6  \\ \hline
13& $D_{46}$ &$\% 99.92$ &3  & & & &                   \\ \hline
16& $D_{58}$ &$\% 99.2$ &5 & & & &                    \\ \hline
17& $D_{62}$ &$\% 99.88$ &5  & & & &                   \\ \hline
\end{tabular}
\end{table}

\newpage
 These calculations suggest the following conjecture.

\begin{conj}\label{con1}
If $\beta_n = \alpha(Z_{p_n})$ and $\gamma_n = \alpha(D_{2p_n})$, then $\lim_{n \rightarrow \infty}\beta_n = \lim_{n \rightarrow \infty}\gamma_n = 1$, where $p_n$ is the $n$-th prime number. 
\end{conj}

In Table \ref{t4}, the percentage of bad parts in some groups of order $3p$, $3 \mid p - 1$ is given.  The calculations given in this table show that the Conjecture \ref{con1} is not valid for groups of order $3p$. Suppose $p$ and $q$ are primes such that  $q < p$ and $q | p - 1$. Let $T_{p,q}$ denote the non-abelian group of order $pq$ and  $q_n$ denote the $n$-th prime number with the property $3 | q_n -1$.
\begin{table}[htp]
\centering
\caption{Percentage of bad parts in $T_{p,3}$.}\label{t4}
\begin{tabular}{|c|c||c|c||c|c|}
\hline
p & $\alpha(G)$ & p & $\alpha(G)$& p & $\alpha(G)$ \\
\hline
$7$ & $\%26$ & $13$ & $\%38$ & $19$ & $\%42$\\
\hline
$31$ & $\%48$ & $37$ & $\%49$ & $43$ & $\%49.6$\\
\hline
\end{tabular}
\end{table}

 By the calculations in Table \ref{t4}, we offer the following conjecture:
\begin{conj}\label{con2}
If $\delta_n = \alpha(T_{q_n, 3})$, then $\lim_{n \rightarrow \infty}\delta_n = 0.5$.
\end{conj}

Suppose  $Irr(Z_{7})$ = $\{1_{Z_7}, \chi_2, \ldots, \chi_{7}\}$ and $Con(Z_{7})$ = $\{e, x_2^{Z_{7}}, \ldots, x_{7}^{Z_{7}}\}$. The cyclic group $Z_7$ has exactly four supercharacter theories $ m(Z_7) $, $ M(Z_7) $, $ \mathcal{C}_1 = (\mathcal{X}_1, \mathcal{K}_1) $ and $ \mathcal{C}_2 = (\mathcal{X}_2, \mathcal{K}_2) $ such that

\begin{eqnarray*}
\mathcal{X}_1  &:=& \{\{1_{Z_7}\}, \{\chi_2, \chi_3, \chi_5 \}, \{\chi_4, \chi_6, \chi_7 \}\},\\
\mathcal{K}_1 &:=& \{\{ e \}, \{{x_2}^{Z_7}, {x_3}^{Z_7}, {x_5}^{Z_7}\}, \{{x_4}^{Z_7}, {x_6}^{Z_7}, {x_7}^{Z_7}\}\},\\
\mathcal{X}_2 &:=& \{\{ 1_{Z_7}\}, \{\chi_2, \chi_7 \}, \{\chi_3, \chi_6 \}, \{\chi_4, \chi_5\}\},\\
\mathcal{K}_2 &:=& \{\{ e \}, \{{x_2}^{Z_7}, {x_7}^{Z_7}\}, \{{x_3}^{Z_7}, {x_6}^{Z_7}\}, \{{x_4}^{Z_7}, {x_5}^{Z_7}\}\}.
\end{eqnarray*}

This shows that each part in $\mathcal{X}$ and $\mathcal{K}$ in a supercharacter theory $(\mathcal{X}, \mathcal{K})$ has size $1, 2, 3$ or $6$. On the other hand, if $p$ is prime and $d$ is the number of divisors of $p-1$, then by \cite[Table 1]{7}, the cyclic group $Z_p$ has exactly $d$ supercharacter theories. 
As a result, the following conjecture is suggested:

\begin{conj}
For each divisor $r$ of $p-1$, there exists only one supercharacter theory $(\mathcal{X}, \mathcal{K})$ of $Z_p$ such that the sizes of all 
non-trivial parts of $\mathcal{X}$ and $\mathcal{K}$ are equal to $r$.  Moreover, if we sort the conjugacy classes and irreducible characters  of $Z_p$ by ATLAS notations\cite{52}, then $\gamma_1(\mathcal{X}) = \gamma_2(\mathcal{K})$.

\end{conj}

It is a well-known result in group theory that  for any positive integer $k$, there are finitely many number of non-isomorphic finite groups with exactly $k$ conjugacy classes. This number is denoted by $f(k)$. Suppose $\Gamma(k) = \{ G_1, G_2, \ldots, G_{f(k)}\}$ denotes a complete set of finite groups such that all members of $\Gamma(k)$ are mutually non-isomorphic and all of them have exactly $k$ conjugacy classes. The supercharacter theory form of $f(k)$ is defined as $n_1^{\alpha_1}n_2^{\alpha_2}\cdots n_s^{\alpha_s}$ where $\alpha_i$, $1 \leq i \leq s$, denotes the number of groups with exactly $k$ conjugacy classes containing $n_i$ supercharacter theories and $f(k) = \sum_{i=1}^s\alpha_i$.
The supercharacter theory form of groups with at most 14 conjugacy classes are recorded in Table 6. 

\begin{center}
TABLE 6. Supercharacter theory form of groups with $\kappa \leq 14$ conjugacy classes.
\end{center}
\begin{longtable}{l|l}
$\kappa$ & Supercharacter Theory Form\\
\hline
3 & $2^2$\\ \hline
4 & $3^3~5^1$\\ \hline
5 & $3^3~5^3~9^2 $\\ \hline
6 & $4^1~5^1~7^1~8^1~9^1~(15)^1~(18)^1~(20)^1$\\ \hline
7 & $3^1~4^1~5^1~7^3~8^2~(11)^1~(20)^3$\\ \hline
8 & $5^2~7^1~(10)^1~(11)^1~(12)^3~(13)^1~(14)^1~(15)^1~(16)^1~(18)^1~(19)^1~(22)^1~(23)^1~$\\
&$(25)^1~(28)^1~(54)^1~(100)^1~(110)^1$\\ \hline
9 & $3^1~7^3~9^1~(10)^1~(12)^1~(13)^1~(15)^1~(18)^1~(19)^1~(21)^1~(22)^2~(32)^2~(36)^1~$\\
&$(43)^1~(40)^1~(45)^3~(49)^1~(65)^1~(128)^2$\\ \hline
10 & $(5)^2~(10)^2~(11)^1~(13)^1~(14)^1~(15)^1~(16)^1~(23)^1~(25)^2~(23)^1~(24)^1~(28)^1~$\\
&$(32)^2~(34)^1~(35)^2~(44)^1~(51)^1~(52)^1$~$(57)^2$~$(58)^3$~$(64)^2$~$(80)^1$~$(83)^2$~$(165)^1$~\\
&$(215)^2$~$(623)^1$\\ \hline
11 & $4^2~8^2~(11)^4~(13)^4~(15)^1~(17)^1~(18)^2~(25)^1~(26)^1~(31)^1~(47)^3~(53)^2~(55)^1~$\\
&$(81)^1~(89)^2~(124)^1~(144)^3~(232)^1~(376)^2$ \\ \hline
12 & $5^1~7^1~(13)^2~(16)^1~(18)^2~(19)^2~(22)^3~(23)^2~(32)^2~(34)^1~(35)^1~(36)^1~(46)^1~$\\
& $(49)^1~(51)^1~(65)^1~(68)^1~(69)^1~(76)^3~(81)^2~(88)^2~(94)^1~(99)^2~(100)^1~(105)^1~$\\
&$(133)^4~(144)^1~(152)^1~(197)^1~(205)^1~(212)^1~(233)^1~(255)^1~(360)^1~(484)^1~$\\
&$(1058)^1$\\ \hline
13 & $3^1~6^1~9^2~(11)^1~(13)^2~(17)^1~(18)^1~(24)^2~(25)^2~(35)^1~(38)^2~(40)^2~(42)^1~$\\
&$(43)^1~(46)^2~(50)^1~
(53)^2~(71)^3~(72)^2~(81)^2~(89)^1~(102)^1~(110)^4~(129)^4~$\\
&$(132)^3~(138)^1~(175)^1~(313)^3~(325)^3$\\ \hline
14 & $(5)^2~(9)^1~(10)^1~(12)^1~(13)^1~(14)^1~(15)^2~(21)^1~(22)^2~(23)^1~(29)^1~(35)^2~$\\
&$(38)^1~(39)^1~(41)^1~(43)^1~(45)^3~(47)^1~(49)^1~(51)^1~(53)^1~(57)^1~(63)^2~(71)^1~$\\
&$(76)^1~(78)^2~(79)^1~(81)^2~(85)^1~(105)^1~(110)^2~(119)^1~(123)^1~(125)^1~(130)^1~$\\
&$(138)^1~(139)^1~(140)^2~(145)^1~(157)^1~(172)^2~(186)^2~(206)^1~(213)^3~(222)^1~$\\
&$(244)^2~(270)^1~(272)^2~(304)^2~(308)^2~(320)^1~(482)^1~(601)^2~(613)^2~(620)^3~$\\
&$(627)^1~(645)^3~(904)^1~(940)^3~(1048)^2~(1324)^3~(2093)^1~(29016)^1$\\ \hline
\end{longtable}

The following conjecture has been suggested by the calculations in Table 6.

\begin{conj}
The number of  supercharacter theories of all members of $\Gamma(k)$ are distinct if and only if $k=6$. In this case, all groups are  $Z_5$, $D_{14}$, $A_5$, $Z_5 : Z_4$, $Z_7 : Z_3$, $S_4$, $D_8$ and $Q_8$.
\end{conj}

Vera--L\'{o}pez and his co-authors\cite{11, 12, 13} classified all finite groups containing up to $14$ conjugacy classes. We apply these classification theorems and our main algorithm to find all supercharacter theories of groups containing up to $14$ conjugacy classes. These  calculations are presented in Table 7. 

\begin{center}
TABLE 7. 
\end{center}
\begin{longtable}{c|c|c|c|c}
\hline 
{\textbar}Con(Group){\textbar} & Group ID & {\textbar}Sup(Group){\textbar} & {\textbar}BadParts{\textbar} & {\textbar}BadPartitionSets{\textbar} \\ \hline 
3 & [3,1] & 2 & 2 & 0 \\ \hline 
3 & [6,1] & 2 & 2 & 0 \\ \hline 
 \hline 
4 & [4,1] & 3 & 4 & 2 \\ \hline 
4 & [10,1] & 3 & 4 & 2 \\ \hline 
4 & [12,3] & 3 & 4 & 2 \\ \hline 
4 & [4,2] & 5 & 0 & 0 \\ \hline 
 \hline 
5 & [14,1] & 3 & 12 & 12 \\ \hline 
5 & [5,1] & 3 & 12 & 12 \\ \hline 
5 & [60,5] & 3 & 8 & 9 \\ \hline 
5 & [20,3] & 5 & 8 & 9 \\ \hline 
5 & [24,12] & 5 & 4 & 6 \\ \hline 
5 & [21,1] & 5 & 4 & 6 \\ \hline 
5 & [8,3] & 9 & 0 & 0 \\ \hline 
5 & [8,4] & 9 & 0 & 0 \\ \hline 
 \hline 
6 & [168,42] & 4 & 16 & 36 \\ \hline 
6 & [18,1] & 5 & 18 & 43 \\ \hline 
6 & [6,2] & 7 & 12 & 36 \\ \hline 
6 & [36,9] & 8 & 8 & 22 \\ \hline 
6 & [12,1] & 9 & 8 & 22 \\ \hline 
6 & [12,4] & 15 & 0 & 0 \\ \hline 
6 & [72,41] & 18 & 0 & 0 \\ \hline 
6 & [18,4] & 20 & 0 & 0 \\ \hline 
 \hline 
7 & [22,1] & 3 & 60 & 200 \\ \hline 
7 & [7,1] & 4 & 54 & 196 \\ \hline 
7 & [120,34] & 5 & 18 & 97 \\ \hline 
7 & [39,1] & 7 & 24 & 124 \\ \hline 
7 & [55,1] & 7 & 24 & 124 \\ \hline 
7 & [360,118] & 7 & 16 & 88 \\ \hline 
7 & [52,3] & 8 & 24 & 120 \\ \hline 
7 & [24,3] & 8 & 40 & 172 \\ \hline 
7 & [42,1] & 11 & 24 & 152 \\ \hline 
7 & [16,8] & 20 & 0 & 0 \\ \hline 
7 & [16,9] & 20 & 0 & 0 \\ \hline 
7 & [16,7] & 20 & 0 & 0 \\ \hline 
 \hline 
8 & [26,1] & 5 & 108 & 858 \\ \hline 
8 & [56,11] & 5 & 108 & 858 \\ \hline 
8 & [720,765] & 7 & 16 & 148 \\ \hline 
8 & [8,1] & 10 & 64 & 750 \\ \hline 
8 & [68,3] & 11 & 48 & 544 \\ \hline 
8 & [48,29] & 12 & 28 & 310 \\ \hline 
8 & [48,28] & 12 & 28 & 310 \\ \hline 
8 & [168,43] & 12 & 40 & 434 \\ \hline 
8 & [660,13] & 13 & 16 & 148 \\ \hline 
8 & [600,150] & 14 & 80 & 761 \\ \hline 
8 & [20,1] & 15 & 16 & 266 \\ \hline 
8 & [80,49] & 16 & 0 & 0 \\ \hline 
8 & [78,1] & 18 & 24 & 326 \\ \hline 
8 & [24,13] & 19 & 24 & 326 \\ \hline 
8 & [48,3] & 22 & 0 & 0 \\ \hline 
8 & [20,4] & 23 & 16 & 266 \\ \hline 
8 & [300,23] & 25 & 16 & 148 \\ \hline 
8 & [8,2] & 28 & 0 & 0 \\ \hline 
8 & [200,44] & 54 & 0 & 0 \\ \hline 
8 & [8,5] & 100 & 0 & 0 \\ \hline 
8 & [48,50] & 110 & 0 & 0 \\ \hline 
 \hline 
9 & g2 & 3 & 92 & 2392 \\ \hline 
9 & [504,156] & 7 & 144 & 3243 \\ \hline 
9 & [120,5] & 7 & 152 & 3640 \\ \hline 
9 & [9,1] & 7 & 168 & 3932 \\ \hline 
9 & [57,1] & 9 & 108 & 2916 \\ \hline 
9 & [336,208] & 10 & 48 & 1266 \\ \hline 
9 & [30,3] & 12 & 96 & 3466 \\ \hline 
9 & [1092,25] & 13 & 48 & 960 \\ \hline 
9 & [72,39] & 15 & 128 & 3634 \\ \hline 
9 & [114,1] & 18 & 72 & 2148 \\ \hline 
9 & [72,15] & 19 & 36 & 1536 \\ \hline 
9 & g1 & 21 & 56 & 1426 \\ \hline 
9 & [1176,215] & 22 & 80 & 2022 \\ \hline 
9 & [18,3] & 22 & 68 & 2108 \\ \hline 
9 & [60,7] & 32 & 16 & 336 \\ \hline 
9 & [960,11357] & 32 & 0 & 0 \\ \hline 
9 & [72,40] & 36 & 0 & 0 \\ \hline 
9 & [9,2] & 40 & 0 & 0 \\ \hline 
9 & [144,182] & 43 & 0 & 0 \\ \hline 
9 & [24,4] & 45 & 0 & 0 \\ \hline 
9 & [24,8] & 45 & 0 & 0 \\ \hline 
9 & [24,6] & 45 & 0 & 0 \\ \hline 
9 & [72,43] & 49 & 0 & 0 \\ \hline 
9 & [36,10] & 65 & 0 & 0 \\ \hline 
9 & [192,1023] & 128 & 0 & 0 \\ \hline 
9 & [192,1025] & 128 & 0 & 0 \\ \hline 
 \hline 
10 & g3 & 5 & 112 & 8192 \\ \hline 
10 & [34,1] & 5 & 480 & 21094 \\ \hline 
10 & [120,35] & 10 & 152 & 12390 \\ \hline 
10 & [10,2] & 10 & 376 & 20725 \\ \hline 
10 & g4 & 11 & 304 & 18784 \\ \hline 
10 & [100,3] & 13 & 240 & 15090 \\ \hline 
10 & [448,179] & 14 & 216 & 14994 \\ \hline 
10 & [28,1] & 15 & 192 & 14676 \\ \hline 
10 & [216,153] & 16 & 84 & 4320 \\ \hline 
10 & g5 & 23 & 0 & 0 \\ \hline 
10 & [28,3] & 23 & 144 & 13182 \\ \hline 
10 & [96,64] & 24 & 112 & 12432 \\ \hline 
10 & [136,12] & 25 & 128 & 11256 \\ \hline 
10 & [150,6] & 25 & 144 & 11700 \\ \hline 
10 & [40,3] & 28 & 128 & 11256 \\ \hline 
10 & [42,2] & 32 & 96 & 9896 \\ \hline 
10 & [48,30] & 32 & 128 & 11256 \\ \hline 
10 & [588,33] & 34 & 96 & 4584 \\ \hline 
10 & [54,6] & 35 & 136 & 10676 \\ \hline 
10 & [54,5] & 35 & 136 & 10676 \\ \hline 
10 & [96,71] & 44 & 48 & 2636 \\ \hline 
10 & [160,234] & 51 & 0 & 0 \\ \hline 
10 & [48,48] & 52 & 0 & 0 \\ \hline 
10 & [100,12] & 57 & 96 & 5056 \\ \hline 
10 & [16,6] & 57 & 0 & 0 \\ \hline 
10 & [96,72] & 58 & 48 & 2636 \\ \hline 
10 & [784,162] & 58 & 0 & 0 \\ \hline 
10 & [96,70] & 58 & 48 & 2636 \\ \hline 
10 & [40,12] & 64 & 0 & 0 \\ \hline 
10 & [96,227] & 64 & 0 & 0 \\ \hline 
10 & [54,8] & 80 & 0 & 0 \\ \hline 
10 & [16,3] & 83 & 0 & 0 \\ \hline 
10 & [16,4] & 83 & 0 & 0 \\ \hline 
10 & [16,13] & 165 & 0 & 0 \\ \hline 
10 & [16,12] & 215 & 0 & 0 \\ \hline 
10 & [16,11] & 215 & 0 & 0 \\ \hline 
10 & [100,11] & 623 & 0 & 0 \\ \hline 
 \hline 
11 & [38,1] & 4 & 1008 & 115959 \\ \hline 
11 & [11,1] & 4 & 990 & 115921 \\ \hline 
11 & [720,763] & 8 & 168 & 34604 \\ \hline 
11 & [116,3] & 8 & 504 & 88928 \\ \hline 
11 & g7 & 11 & 288 & 34056 \\ \hline 
11 & [1344,814] & 11 & 176 & 21962 \\ \hline 
11 & [1344,11686] & 11 & 176 & 21962 \\ \hline 
11 & [1512,779] & 11 & 496 & 75799 \\ \hline 
11 & [75,2] & 13 & 256 & 40208 \\ \hline 
11 & [336,114] & 13 & 312 & 53894 \\ \hline 
11 & [155,1] & 13 & 648 & 104700 \\ \hline 
11 & [203,1] & 13 & 648 & 104700 \\ \hline 
11 & [110,1] & 15 & 752 & 114070 \\ \hline 
11 & [171,3] & 17 & 336 & 74820 \\ \hline 
11 & [720,764] & 18 & 224 & 45882 \\ \hline 
11 & [186,1] & 18 & 360 & 75020 \\ \hline 
11 & g6 & 25 & 192 & 23424 \\ \hline 
11 & [432,734] & 26 & 84 & 12306 \\ \hline 
11 & [320,1635] & 31 & 128 & 21560 \\ \hline 
11 & [32,19] & 47 & 96 & 6688 \\ \hline 
11 & [32,20] & 47 & 96 & 6688 \\ \hline 
11 & [32,18] & 47 & 96 & 6688 \\ \hline 
11 & [192,184] & 53 & 256 & 61636 \\ \hline 
11 & [192,185] & 53 & 256 & 61636 \\ \hline 
11 & [200,40] & 55 & 0 & 0 \\ \hline 
11 & [392,38] & 81 & 0 & 0 \\ \hline 
11 & [27,4] & 89 & 0 & 0 \\ \hline 
11 & [27,3] & 89 & 0 & 0 \\ \hline 
11 & [108,17] & 124 & 0 & 0 \\ \hline 
11 & [32,6] & 144 & 0 & 0 \\ \hline 
11 & [32,7] & 144 & 0 & 0 \\ \hline 
11 & [32,8] & 144 & 0 & 0 \\ \hline 
11 & [96,204] & 232 & 0 & 0 \\ \hline 
11 & [32,43] & 376 & 0 & 0 \\ \hline 
11 & [32,44] & 376 & 0 & 0 \\ \hline 
 \hline 
12 & g14 & 5 & 288 & 58212 \\ \hline 
12 & g11 & 7 & 656 & 381020 \\ \hline 
12 & [336,209] & 13 & 584 & 358258 \\ \hline 
12 & [42,5] & 13 & 1320 & 668439 \\ \hline 
12 & g12 & 16 & 288 & 93764 \\ \hline 
12 & [360,120] & 18 & 76 & 31298 \\ \hline 
12 & g13 & 18 & 912 & 417606 \\ \hline 
12 & [240,90] & 19 & 192 & 59192 \\ \hline 
12 & [240,89] & 19 & 192 & 59192 \\ \hline 
12 & [84,11] & 22 & 696 & 425118 \\ \hline 
12 & g8 & 22 & 1200 & 521630 \\ \hline 
12 & [36,3] & 22 & 1248 & 591102 \\ \hline 
12 & [1920,240993] & 23 & 164 & 48724 \\ \hline 
12 & [960,11358] & 23 & 128 & 61560 \\ \hline 
12 & [222,1] & 32 & 648 & 440232 \\ \hline 
12 & [12,2] & 32 & 576 & 575908 \\ \hline 
12 & [3420,144] & 34 & 192 & 91344 \\ \hline 
12 & [36,1] & 35 & 480 & 416154 \\ \hline 
12 & [30,2] & 36 & 760 & 535850 \\ \hline 
12 & [24,1] & 46 & 256 & 346836 \\ \hline 
12 & [72,19] & 49 & 512 & 443456 \\ \hline 
12 & [36,4] & 51 & 168 & 101808 \\ \hline 
12 & [96,3] & 65 & 0 & 0 \\ \hline 
12 & [168,49] & 68 & 288 & 109428 \\ \hline 
12 & [126,9] & 69 & 320 & 117640 \\ \hline 
12 & g10 & 76 & 256 & 119632 \\ \hline 
12 & [36,11] & 76 & 0 & 0 \\ \hline 
12 & [12,5] & 76 & 144 & 120036 \\ \hline 
12 & [108,37] & 81 & 64 & 8112 \\ \hline 
12 & [72,44] & 81 & 136 & 69676 \\ \hline 
12 & [384,18134] & 88 & 0 & 0 \\ \hline 
12 & [384,18135] & 88 & 0 & 0 \\ \hline 
12 & [60,8] & 94 & 72 & 74496 \\ \hline 
12 & [384,592] & 99 & 0 & 0 \\ \hline 
12 & [384,591] & 99 & 0 & 0 \\ \hline 
12 & [24,5] & 100 & 0 & 0 \\ \hline 
12 & [72,45] & 105 & 128 & 86112 \\ \hline 
12 & [48,15] & 133 & 0 & 0 \\ \hline 
12 & [48,16] & 133 & 0 & 0 \\ \hline 
12 & [48,18] & 133 & 0 & 0 \\ \hline 
12 & [48,17] & 133 & 0 & 0 \\ \hline 
12 & [216,161] & 144 & 0 & 0 \\ \hline 
12 & [24,7] & 152 & 0 & 0 \\ \hline 
12 & [36,7] & 197 & 0 & 0 \\ \hline 
12 & [96,203] & 205 & 0 & 0 \\ \hline 
12 & [144,120] & 212 & 0 & 0 \\ \hline 
12 & [1620,419] & 233 & 0 & 0 \\ \hline 
12 & [36,13] & 255 & 0 & 0 \\ \hline 
12 & [24,14] & 360 & 0 & 0 \\ \hline 
12 & [144,187] & 484 & 0 & 0 \\ \hline 
12 & [1296,3523] & 1058 & 0 & 0 \\ \hline 
 \hline 
13 & [46,1] & 3 & 4092 & 4213594 \\ \hline 
13 & [13,1] & 6 & 4020 & 4213372 \\ \hline 
13 & [93,1] & 9 & 1980 & 3367930 \\ \hline 
13 & [253,1] & 9 & 1980 & 3367930 \\ \hline 
13 & [148,3] & 11 & 2016 & 3368736 \\ \hline 
13 & [205,1] & 13 & 2880 & 3960896 \\ \hline 
13 & [150,5] & 13 & 2552 & 4109564 \\ \hline 
13 & [301,1] & 17 & 2916 & 4035570 \\ \hline 
13 & [258,1] & 18 & 1512 & 3012604 \\ \hline 
13 & [720,409] & 24 & 432 & 592604 \\ \hline 
13 & [333,3] & 24 & 2016 & 3571488 \\ \hline 
13 & [310,1] & 25 & 2256 & 3196080 \\ \hline 
13 & [328,12] & 25 & 1920 & 3480620 \\ \hline 
13 & [160,199] & 35 & 0 & 0 \\ \hline 
13 & [324,160] & 38 & 528 & 882360 \\ \hline 
13 & [1320,133] & 38 & 400 & 1169056 \\ \hline 
13 & [1944,2290] & 40 & 1044 & 535050 \\ \hline 
13 & [1944,2289] & 40 & 1044 & 535050 \\ \hline 
13 & [1440,5841] & 42 & 0 & 0 \\ \hline 
13 & [120,38] & 43 & 1040 & 2034000 \\ \hline 
13 & [216,86] & 46 & 896 & 2642808 \\ \hline 
13 & [156,7] & 46 & 1152 & 3635061 \\ \hline 
13 & [1053,51] & 50 & 1008 & 1493856 \\ \hline 
13 & [576,8652] & 53 & 608 & 761056 \\ \hline 
13 & [100,10] & 53 & 896 & 1241736 \\ \hline 
13 & [40,6] & 71 & 288 & 259112 \\ \hline 
13 & [40,4] & 71 & 288 & 259112 \\ \hline 
13 & [40,8] & 71 & 288 & 259112 \\ \hline 
13 & [960,11359] & 72 & 384 & 609896 \\ \hline 
13 & [600,148] & 72 & 512 & 2172408 \\ \hline 
13 & [162,15] & 81 & 1104 & 1312350 \\ \hline 
13 & [162,13] & 81 & 1104 & 1312350 \\ \hline 
13 & [216,87] & 89 & 0 & 0 \\ \hline 
13 & [162,11] & 102 & 816 & 919596 \\ \hline 
13 & [1000,86] & 110 & 0 & 0 \\ \hline 
13 & [96,190] & 110 & 0 & 0 \\ \hline 
13 & [96,193] & 110 & 0 & 0 \\ \hline 
13 & [96,191] & 110 & 0 & 0 \\ \hline 
13 & [192,1493] & 129 & 0 & 0 \\ \hline 
13 & [192,1491] & 129 & 0 & 0 \\ \hline 
13 & [192,1494] & 129 & 0 & 0 \\ \hline 
13 & [192,1492] & 129 & 0 & 0 \\ \hline 
13 & [162,22] & 132 & 0 & 0 \\ \hline 
13 & [162,20] & 132 & 0 & 0 \\ \hline 
13 & [162,21] & 132 & 0 & 0 \\ \hline 
13 & [162,19] & 138 & 0 & 0 \\ \hline 
13 & [400,206] & 175 & 64 & 51232 \\ \hline 
13 & [64,35] & 313 & 128 & 102200 \\ \hline 
13 & [64,33] & 313 & 128 & 102200 \\ \hline 
13 & [64,36] & 313 & 128 & 102200 \\ \hline 
13 & [64,32] & 325 & 0 & 0 \\ \hline 
13 & [64,34] & 325 & 0 & 0 \\ \hline 
13 & [64,37] & 325 & 0 & 0 \\ \hline 
 \hline 
14 & g20 & 5 & 1752 & 5731764 \\ \hline 
14 & g17 & 5 & 2632 & 13381474 \\ \hline 
14 & g22 & 9 & 2388 & 8785964 \\ \hline 
14 & [50,1] & 10 & 6160 & 27476909 \\ \hline 
14 & g23 & 12 & 1328 & 4311840 \\ \hline 
14 & [14,2] & 13 & 7236 & 27615724 \\ \hline 
14 & [164,3] & 14 & 3960 & 22166120 \\ \hline 
14 & g15 & 15 & 3424 & 16996376 \\ \hline 
14 & [44,1] & 15 & 3840 & 22147610 \\ \hline 
14 & g21 & 21 & 0 & 0 \\ \hline 
14 & [720,766] & 22 & 1112 & 6206500 \\ \hline 
14 & [240,91] & 22 & 1824 & 15390702 \\ \hline 
14 & [44,3] & 23 & 2960 & 20070440 \\ \hline 
14 & [294,1] & 29 & 3024 & 19661348 \\ \hline 
14 & g24 & 35 & 0 & 0 \\ \hline 
14 & [410,1] & 35 & 4512 & 21859208 \\ \hline 
14 & [240,189] & 38 & 96 & 639362 \\ \hline 
14 & [1344,816] & 39 & 2112 & 15007172 \\ \hline 
14 & g18 & 41 & 0 & 0 \\ \hline 
14 & [108,15] & 43 & 1728 & 12086940 \\ \hline 
14 & [110,2] & 45 & 3008 & 19628768 \\ \hline 
14 & [78,2] & 45 & 2400 & 18340208 \\ \hline 
14 & [392,36] & 45 & 3456 & 22890564 \\ \hline 
14 & [1920,241001] & 47 & 0 & 0 \\ \hline 
14 & [104,3] & 49 & 3072 & 22597332 \\ \hline 
14 & [48,33] & 51 & 2272 & 20063952 \\ \hline 
14 & [2420,43] & 53 & 960 & 4740580 \\ \hline 
14 & g19 & 57 & 1120 & 3289320 \\ \hline 
14 & [480,1188] & 63 & 1520 & 8453444 \\ \hline 
14 & [648,703] & 63 & 0 & 0 \\ \hline 
14 & [288,1025] & 71 & 1188 & 4267596 \\ \hline 
14 & [300,25] & 76 & 608 & 2367488 \\ \hline 
14 & g16 & 78 & 2880 & 8390856 \\ \hline 
14 & [300,24] & 78 & 1152 & 10776156 \\ \hline 
14 & [84,1] & 79 & 1152 & 10776156 \\ \hline 
14 & [200,43] & 81 & 1184 & 4567684 \\ \hline 
14 & [500,21] & 81 & 1792 & 7771408 \\ \hline 
14 & [48,32] & 85 & 1088 & 13808484 \\ \hline 
14 & [104,12] & 105 & 576 & 4187988 \\ \hline 
14 & [320,1582] & 110 & 0 & 0 \\ \hline 
14 & [320,1581] & 110 & 0 & 0 \\ \hline 
14 & [1920,241000] & 119 & 0 & 0 \\ \hline 
14 & [432,520] & 123 & 96 & 801624 \\ \hline 
14 & g9 & 125 & 0 & 0 \\ \hline 
14 & [384,5] & 130 & 0 & 0 \\ \hline 
14 & [1280,1116311] & 138 & 768 & 2059168 \\ \hline 
14 & [240,192] & 139 & 768 & 3811584 \\ \hline 
14 & [384,4] & 140 & 0 & 0 \\ \hline 
14 & [384,6] & 140 & 0 & 0 \\ \hline 
14 & [294,14] & 145 & 1296 & 5737440 \\ \hline 
14 & [84,7] & 157 & 288 & 2705980 \\ \hline 
14 & [192,955] & 172 & 0 & 0 \\ \hline 
14 & [192,956] & 172 & 0 & 0 \\ \hline 
14 & [1280,1116310] & 186 & 0 & 0 \\ \hline 
14 & [1280,1116312] & 186 & 0 & 0 \\ \hline 
14 & [32,15] & 206 & 1024 & 6202512 \\ \hline 
14 & [96,195] & 213 & 64 & 109120 \\ \hline 
14 & [96,185] & 213 & 64 & 109120 \\ \hline 
14 & [96,187] & 213 & 64 & 109120 \\ \hline 
14 & [32,11] & 222 & 256 & 237632 \\ \hline 
14 & [80,29] & 244 & 0 & 0 \\ \hline 
14 & [80,33] & 244 & 0 & 0 \\ \hline 
14 & [288,1026] & 270 & 0 & 0 \\ \hline 
14 & [32,10] & 272 & 0 & 0 \\ \hline 
14 & [32,9] & 272 & 0 & 0 \\ \hline 
14 & [32,14] & 304 & 1024 & 6202512 \\ \hline 
14 & [32,13] & 304 & 1024 & 6202512 \\ \hline 
14 & [80,31] & 308 & 0 & 0 \\ \hline 
14 & [80,34] & 308 & 0 & 0 \\ \hline 
14 & [50,4] & 320 & 0 & 0 \\ \hline 
14 & [32,42] & 482 & 0 & 0 \\ \hline 
14 & [128,145] & 601 & 256 & 606592 \\ \hline 
14 & [128,144] & 601 & 256 & 606592 \\ \hline 
14 & [128,138] & 613 & 0 & 0 \\ \hline 
14 & [128,139] & 613 & 0 & 0 \\ \hline 
14 & [32,41] & 620 & 0 & 0 \\ \hline 
14 & [32,40] & 620 & 0 & 0 \\ \hline 
14 & [32,39] & 620 & 0 & 0 \\ \hline 
14 & [200,42] & 627 & 0 & 0 \\ \hline 
14 & [384,5871] & 645 & 0 & 0 \\ \hline 
14 & [384,5868] & 645 & 0 & 0 \\ \hline 
14 & [384,5870] & 645 & 0 & 0 \\ \hline 
14 & [32,33] & 904 & 0 & 0 \\ \hline 
14 & [32,32] & 940 & 0 & 0 \\ \hline 
14 & [32,30] & 940 & 0 & 0 \\ \hline 
14 & [32,31] & 940 & 0 & 0 \\ \hline 
14 & [32,29] & 1048 & 0 & 0 \\ \hline 
14 & [32,28] & 1048 & 0 & 0 \\ \hline 
14 & [32,27] & 1324 & 0 & 0 \\ \hline 
14 & [32,35] & 1324 & 0 & 0 \\ \hline 
14 & [32,34] & 1324 & 0 & 0 \\ \hline 
14 & [500,23] & 2093 & 0 & 0 \\ \hline 
14 & [294,13] & 29016 & 0 & 0 \\ \hline 
\end{longtable}

\vskip 3mm

\noindent{\textbf{Acknowledgment.}} The authors are indebted to professor Thomas Breuer for some critical discussion on primitive groups and the method he has suggested for introducing them to GAP.

\end{document}